\documentclass[12pt]{article}
\usepackage{listings}
\usepackage{url}
\usepackage{array}
\usepackage{xcolor}
\usepackage{amsmath}
\usepackage{amsfonts}
\usepackage{amsthm}
\usepackage{enumerate}
\usepackage{geometry}

\usepackage{amssymb,latexsym}

\ifdefined\directlua
\usepackage{fontspec}
\usepackage{polyglossia}
\setmainlanguage{english}
\usepackage{microtype}
\usepackage{pdftexcmds}
\makeatletter
\ifcase\pdf@shellescape
\relax\or
\begingroup
  \catcode`\%=12\relax
  \gdef\sformat{"Date: 
\endgroup
\directlua{
 local cmd="git show -s --format='"..\sformat.."'"
 local r=io.popen(cmd):read("*a")
 if (r) then
      tex.print("\string\\def\string\\COMMIT{"..r.."}") 
 end
 }
\or
\relax\fi
\makeatother
\ifdefined\COMMIT
        \usepackage{background}
        \backgroundsetup{%
         pages=all, placement=bottom,angle=0,scale=2,%
         vshift=20pt,%
         contents={\COMMIT}}
\fi
\else
\fi

\newcommand{\SL}{\mathrm{SL}}
\newcommand{\wt}{\mathrm{wt}\,}
\newcommand{\FF}{\mathbb F}
\newcommand{\QQ}{\mathbb Q}
\newcommand{\RR}{\mathbb R}
\newcommand{\SSS}{\mathbb S}
\newcommand{\KK}{\mathbb K}
\newcommand{\PP}{\mathbb P}
\newcommand{\fq}{\mathfrak q}
\newcommand{\fb}{\mathfrak b}
\newcommand{\fC}{\mathfrak C}
\newcommand{\fR}{\mathfrak R}
\newcommand{\cD}{\mathcal D}
\newcommand{\cV}{\mathcal V}
\newcommand{\cS}{\mathcal S}
\newcommand{\cQ}{\mathcal Q}
\newcommand{\cL}{\mathcal L}
\newcommand{\cI}{\mathcal I}
\newcommand{\cP}{\mathcal P}
\newcommand{\GG}{\mathbb G}
\newcommand{\DD}{\mathbb D}
\newcommand{\VV}{\mathbb V}
\newcommand{\LL}{\mathbb L}
\newcommand{\cC}{\mathcal C}
\newcommand{\cF}{\mathcal F}
\newcommand{\Rad}{\mathrm{Rad}\,}
\newcommand{\fa}{\mathfrak a}
\newcommand{\fx}{\mathfrak x}
\newcommand{\fv}{\mathfrak v}
\newcommand{\fp}{\mathfrak{p}}
\newcommand{\Cp}{p}
\newcommand{\Cq}{q}
\newcommand{\PG}{\mathrm{PG}}
\newcommand{\PGL}{\mathrm{PGL}}
\newcommand{\bee}[1]{\mathbf{e}_{#1}}
\newcommand{\GL}{\mathrm{GL}}
\newcommand{\fH}{\mathfrak H}
\newcommand{\cod}{\operatorname{codim}}
\newcommand{\diam}{\operatorname{diam}}
\newcommand{\Hom}{\mathrm{Hom}}
\newcommand{\Img}{\mathrm{Img}}
\newcommand{\Aut}{\mathrm{Aut}\,}
\newcommand{\Tr}{\mathrm{Tr}\,}
\newcommand{\rank}{\mathrm{rank}\,}
\newcommand{\cha}{\mathrm{char}\,}
\newcommand{\llbracket}{[\hspace{-0.5mm}[}
\newcommand{\rrbracket}{]\hspace{-0.5mm}]}
\newcommand{\Der}{\mathrm{Der}}
\newcommand{\ndiv}{\not\hspace{-0.5mm}|~}
\newcommand{\ddiv}{|}

\newtheorem{theorem}{Theorem}[section]
\newtheorem{lemma}[theorem]{Lemma}
\newtheorem{corollary}[theorem]{Corollary}
\newtheorem{prop}[theorem]{Proposition}

\theoremstyle{definition}

\newcommand{\cM}{\mathcal M}
\newcommand{\ovcM}{\overline{\mathcal M}}
\newcommand{\cX}{\mathcal X}
\author{I.~Cardinali, L.~Giuzzi, A.~Pasini}
\title{On the $1$-cohomology of $\SL(n,\KK)$ on the dual
  of its adjoint module}
\begin{document}
\maketitle
\begin{abstract}
Given a field $\KK$, for any $n\geq 3$ the first cohomology group $H^1(G_n,A^*_n)$ of the special linear group $G_n = \SL(n,\KK)$ over the dual $A^*_n$ of its adjoint module $A_n$ is isomorphic to the space $\Der(\KK)$ of the derivations of $\KK$, except possibly when $|\KK| \in \{2, 4\}$ and $n$ is even. This fact is mentioned by S. Smith and H. V\"{o}lklein in their paper {\em A geometric presentation for the adjont module of $\SL_3(k)$} (J. Algebra 127 (1989), 127--138). They claim that when $|\KK| > 9$ this fact follows from the main result of V\"{o}lklein's paper {\em The 1-cohomology of the adjoint module of a Chevalley group} (Forum Math. 1 (1989), 1--13), but say nothing that can help the reader to deduce it from that result. When $|\KK| \leq 9$ they obtain the isomorphism $H^1(G_n,A^*_n) \cong \Der(\KK)$ by means of other results from homological algebra, which however miss the case $|\KK| \in\{2, 4\}$ with $n $ even. In the present paper we shall provide a straightforward proof of the isomorphism $H^1(G_n,A^*_n) \cong \Der(\KK)$ under the hypothesis $n > 3$. Our proof also covers the above mentioned missing case.

\end{abstract}

\section{Introduction}

Let $\KK$ be a field and $3\leq n<\infty$. From now on we write $G_n$ for $\SL(n,\KK)$ and $\cM_n =\cM_n(\KK)$ for the space of all $n\times n$ matrices with entries in $\KK$. Regarding the elements of $G_n$ as non-singular matrices, the group $G_n$ acts as follows on $\cM_n$:
\begin{equation}\label{adj action}
\forall x\in \cM_n, ~\forall g\in G:\quad x\cdot g~:=~g^{-1}xg.
\end{equation}
This is the so-called {\em adjoint action} of $G_n$ on $\cM_n$. Clearly, $G_n$ preserves traces, explicitly: $\Tr(x\cdot g) = \Tr(x)$ for every $x\in \cM_n$ and every $g\in G_n$. In particular, $G_n$ stabilizes the subspace $A_n := \{ x\in\cM_{n} \colon \mathrm{Tr}(x)=0 \}$ of $\cM_n$ formed by the traceless matrices of $\cM_n$. So, $A_n$ is a $G_n$-module, called the {\em adjoint module} for $G_n$. Formula \eqref{adj action} (but with $x\in A_n$) describes the action of $G_n$ on $A_n$. Accordingly, $G_n$ acts as follows on the dual $A^*_n$ of $A_n$:
\[\forall \alpha\in A_n^*, ~\forall g\in G_n:\quad \alpha\cdot g~: ~ a\in A_n ~ \rightarrow \alpha(gag^{-1}).\]
Turning back to $\cM_n$,  let $\cI_n = \langle I_n\rangle$ be the $1$-dimensional subspace of $\cM_n$ generated by the identity matrix $I_n$, namely the subspace formed by scalar matrices. Obviously $G_n$ stabilizes $\cI_n$. Hence the quotient $\cM_n/\cI_n$ inherits a $G_n$-module structure from $\cM_n$, with $G_n$ acting as follows on it:
\[\forall (x+\cI_n)\in \cM_n/\cI_n,\forall g\in G:\quad (x+\cI_n)\cdot g~ =~g^{-1}xg + \cI_n.\]

\subsection{The isomorphism $A_n^* \cong \cM_n/\cI_n$}

As Smith and V\"{o}lklein remark in \cite{SV}, the $G_n$-modules $A^*_n$ and $\cM_n/\cI_n$ are isomorphic (as modules). Since this isomorphism is crucial for what we are going to do in this paper, it is worth to spend a few words to explain how to construct it.

Let $\tau: \cM_n/\cI_n \times A_n\rightarrow \KK$ be the bilinear mapping defined as follows:
\[\forall [x]\in \cM_n/\cI_n, ~\forall a\in A_n: ~~~ \tau([x], a) ~ = ~ \Tr(xa), ~ \text{ for } x\in [x],\]
where $xa$ is the usual row-times-column product of $x$ and $a$. Since for every $t\in \KK$ we have $\Tr((x+tI_n)a) = \Tr(xa) + t\Tr(a) = \Tr(xa)$ (because $\Tr(a) = 0$, as $a\in A_n$), this definition does not depend on the particular choice of the representative $x$ of the coset $[x] = x+\cI_n$. So $\tau$ is well defined. It is not difficult to prove also that $\tau$ is non-degenerate, namely both its radicals
\[\{[x]\in \cM_n/\cI_n~:~ \tau([x], a) = 0,~\forall a\in A_n\} \]
and
\[ \{a\in A_n~:~ \tau([x], a) = 0,~\forall [x]\in \cM_n/\cI_n\}\]
are null. Hence for every non-zero $[x]\in\cM_n/\cI_n$ the mapping $\alpha_{[x]}:A_n\rightarrow \KK$ which maps every $a\in A_n$ onto $\tau([x], a)$ is a non-zero linear functional of $A_n$ and every linear functional $\alpha\in A_n^*$ can be obtained in this way for exactly one element $[x]_{\alpha}$ of $\cM_n/\cI_n$. So, $A_n^* \cong \cM_n/\cI_n$.

The isomorphism $A_n^* \cong \cM_n/\cI_n$, which we have now established as an isomorphism of vector spaces, is actually an isomorphism of $G_n$-modules. Indeed let $\alpha = \alpha_{[x]}$ be the linear functional corresponding to $[x] = x+\cI_n$. For every $g\in G_n$, the linear functional $\alpha\cdot g$ maps every  $a\in A_n$ onto $\alpha(gag^{-1}) = \Tr(x(gag^{-1})) = \Tr(g^{-1}(xgag^{-1})g) = \Tr((g^{-1}xg)a)$, which is just the value taken by $\alpha_{[g^{-1}xg]} = \alpha_{[x]\cdot g}$ at $a$. In short, $\alpha_{[x]}\cdot g = \alpha_{[x]\cdot g}$, which proves our claim.

Henceforth we shall always identify $A_n^*$ with $\cM_n/\cI_n$, thus regarding the elements of $A_n^*$ as elements of $\cM_n/\cI_n$. Accordingly, if $\alpha\in A_n^*$ and $a\in A_n$ then $\tau(\alpha, a)$ (which, by a little abuse, can be written as $\Tr(\alpha\cdot a)$), is the value taken by $\alpha$ at $a$.

\subsection{The derivations of $\KK$}\label{derivazioni}

A \emph{derivation} of $\KK$ is a function
  $d:\KK\to\KK$ such that
  \begin{enumerate}[{\rm (D1)}]
  \item\label{D1} $d(x+y)= d(x)+ d(y)$;
  \item\label{D2} $d(xy)=d(x)y+x d(y)$.
  \end{enumerate}
Note that (D\ref{D2}) implies $d(1) = 0$. Hence $d(k\cdot 1) = 0$ for every integer $k$ and $d(t^{-1}) = -d(t)/t^2$ for every $t\neq 0$. Consequently, $d(t) = 0$ for every $t$ in the minimal subfield $\KK_0$ of $\KK$. Morever, $d(t^m) = mt^{m-1}d(t)$ for every positive integer $m$. Therefore, if $\cha(\KK) = p >  0$ then $d(t^p) = 0$ for every $t\in\KK$. In this case $\KK^p$ is the maximum subfield of $\KK$ such that all derivations of $\KK$ induce the null derivation on it \cite[chp. II, \S\S 12 and 17]{ZS}.

On the other hand, let $\cha(\KK) =0$ and suppose that $t\in \KK$ is algebraic over $\KK_0$. Then $d(t) =0$. In this case the algebraic closure of $\KK_0$ in $\KK$ is the maximal subfield of $\KK$ such that all derivations of $\KK$ are null on it \cite[chp. II, \S\S 12 and 17]{ZS}).  Accordingly,
\begin{prop}
The field $\KK$ admits only the null derivation if and only if it is either perfect of positive characteristic or algebraic over the rationals.
\end{prop}

\subsection{The result to be proved in this paper}

The derivations of $\KK$ form a $\KK$-vector space, denoted by $\Der(\KK)$, indeed a subspace of the $\KK$-vector space of all mappings from $\KK$ to $\KK$. Recall that, for every group $X$ and every $X$-module $V$ defined over $\KK$, the first cohomology group $H^1(X,V)$ of $X$ over $V$ is actually a $\KK$-vector space. The following is well known (V\"{o}lklein \cite{V}, also Taussky and Zassenhaus \cite{TZ}):

\begin{theorem}\label{noto}
Let $|\KK| > 4$. Then $H^1(G_n, A_n) ~\cong ~\Der(\KK)\oplus(A_n\cap \cI_n)$.
\end{theorem}

Smith and V\"{o}lklein \cite[Lemma (2.2)]{SV} also prove the following:

\begin{theorem}\label{ri-noto}
We have $H^1(G_3, A^*_3) ~\cong ~\Der(\KK)$ for any choice of the field $\KK$.
\end{theorem}

The core of the proof sketched in \cite{SV} for Theorem~\ref{ri-noto}  amounts to the claim that when $|\KK| > 9$ the isomorphism $H^1(G_n, A^*_n) ~\cong ~\Der(\KK)$ follows from Theorem \ref{noto}. The cases with $|\KK| \leq 9$ are dealt  with the help of other results from homological algebra. As we shall see in a few lines, the above claim is obviously true when $\cha(\KK)$ is either $0$ or positive but prime to $n$, even with the hypothesis $|\KK| > 9$ replaced by $|\KK| > 4$. In contrast, when $\cha(\KK)$ is positive and divides $n$ (and $|\KK| > 9$ as assumed in \cite{SV}) we do not see any easy way to obtain the statement of Theorem \ref{ri-noto} from Theorem \ref{noto}. Regretfully, the authors of \cite{SV} say nothing that can help the reader to understand this point.

In this paper, leaving aside the case $n = 3$, which we believe it has been fixed in \cite{SV}, we shall focus on the case $n > 3$, providing a straightforward (but rather laborious) proof of the following.

\begin{theorem}\label{main-main}
Let $n > 3$. Then $H^1(G_n,A_n^*) \cong\mathrm{Der}(\KK)$ for every choice of the field $\KK$.
\end{theorem}

As said above, when $|\KK| > 4$ and $\cha(\KK)$ is either $0$ or positive but  prime to $n$ then the conclusion of Theorem \ref{main-main} immediately follows from Theorem \ref{noto}. Indeed, let $p=\mathrm{char}(\KK)$. As $\Tr(I_n) = n$, we have $I_n \in A_n$ if and only if $p$ is positive and divides $n$. Hence if $p$ is positive and prime to $n$ or $p = 0$ then $\cM_n/\cI_n\cong A_n$. Note also that $\cM_n/\cI_n$ admits no non-trivial proper submodule while, if $I_n\in A_n$, then $\cI_n$ is a proper submodule of $A_n$. Consequently,

\begin{lemma}\label{pchar}
We have $A_n^*\cong A_n$ (as $G_n$-modules) if and only if either $p = 0$ or $p > 0$ and $(p,n) = 1$.
\end{lemma}

By Lemma \ref{pchar} and Theorem \ref{noto} and since $I_n\in A_n$ if and only if $0 < p \mid n$, the statement of Theorem \ref{main-main} holds true when $|\KK| > 9$ and either $p = 0$ or $p > 0$ is prime to $n$. However, in the proof we shall give in this paper for Theorem \ref{main-main} we will make no use of Theorem \ref{noto}, not even in the favorable case where $p$ is null or positive and prime to $n$.

We fix some notation. Recall that the elements of the first cohomology group $H^1(G_n, A_n^*)$ of $G_n$ are the cosets of the subspace $B^1(G_n, A_n^*)$ of $1$-coboundaries in the space $Z^1(G_n, A_n^*)$ of $1$-cocycles. If $f\in Z^1(G_n, A_n^*)$ we put 
\[\llbracket f\rrbracket ~:= ~ f+B^1(G_n, A_n^*).\] 
If $\phi = \llbracket f\rrbracket$ for a $1$-cocycle $f\in Z^1(G_n, A^*_n)$ we say that $f$ is a {\em representative} of $\phi$ in $Z^1(G_n,A^*_n)$ (also a {\em representative} of $\phi$ for short if the context makes it clear that $f \in Z^1(G_n,A^*_n)$). Moreover if $F\in C^1(\cM_n,G_n)$ is such that 
\[f(g)~ = ~ [F(g)] ~ = ~ F(g)+ \cI_n ~ \in \cM_n/\cI_n \cong A^*_n, ~~ \forall g\in G_n\]
then we say that $F$ is a {\em representative} of $f$ in the space $C^1(\cM_n,G_n)$ of $1$-cochains of $G_n$ over $\cM_n$ (also a {\em representative} of $\phi$ for short if the context makes it clear that $F\in C^1(\cM_n,G_n)$).  f $F\in C^1(\cM_n,G_n)$ represents $f\in Z^1(A^*_n,G_n)$ we write$f= [F]$. 

Henceforth we denote by $\bee{kh}$ the $n\times n$ matrix with all null entries but the $(k,h)$-entry, which is $1$. For $k \neq h$ and $s\in \KK$ we put $u_{kh}(s) := I_n + s\bee{kh}$. Recall that $U_{ks} := \{u_{kh}(s)\}_{s\in \KK}$ is an abelian subgroup of $G_n$, called a {\em root subgroup}. The subgroups $U^+ := \langle U_{kh}\rangle_{k < h}$ and $U^- := \langle U_{kh}\rangle_{k > h}$ are maximal nilpotent subgroups of $G_n$, they meet trivially and their union generates $G_n$. Consequently,

\begin{lemma}\label{main1}
Every $1$-cocycle of $G_n$ over $A^*_n$ is uniquely determined by the value it takes on each of the generators $u_{kh}(s)$ of $G_n$. The same is true for $\cM_n$ and $A_n$.
\end{lemma}

Given a derivation $\alpha$ and an element $g =\sum_{i,j=1}^ng_{ij}\bee{ij}$ of $G_n$, we put $\alpha(g) := \sum_{i,j=1}^n\alpha(g_{ij})\bee{ij}$ ($\in \cM_n$). The next lemma is well known (see e.g. \cite[\S 5]{TZ}). However, for the sake of completeness, we shall recall its proof in Section \ref{TZ}.

\begin{lemma}\label{main2}
For every derivation $\alpha\in \Der(\KK)$, the mapping
\[f_{\alpha} ~:~ g\in G_n ~ \longrightarrow ~ [g^{-1}\alpha(g)] = g^{-1}\alpha(g) + \cI_n \in A^*_n\]
belongs to $Z^1(G_n, A^*_n)$. Moreover, $f_{\alpha}\in B^1(G_n, A^*_n)$ if and only if $\alpha = 0$.
\end{lemma}

With $f_{\alpha}$ as above, we have $f_{\alpha}(u_{kh}(s)) = [\alpha(s)\bee{kh}]$ for every choice of $k\neq h$ and $s\in \KK$. Therefore, in view of Lemma \ref{main1},
\begin{corollary}\label{main3}
With $\alpha$ and $f_{\alpha}$ as in Lemma \ref{main2}, $\llbracket f_{\alpha}\rrbracket$ is the unique element of $H^1(G_n,A^*_n)$ which admits a representative $f \in Z^1(G_n, A^*_n)$ such that $f(u_{kh}(s)) = [\alpha(s)\bee{kh}]$ for every generator $u_{kh}(s)$ of $G_n$.
\end{corollary}

Most of this paper is devoted to the proof of the following.

\begin{lemma}\label{main4}
Let $n > 3$. Then for every $\phi\in H^1(G_n, A^*_n)$ there exists a derivation $\alpha_\phi\in \Der(\KK)$ and a representative $f$ of $\phi$ in $Z^1(G_n, A^*_n)$ such that $\phi=\llbracket f\rrbracket$ and $f(u_{kh}(s)) = [\alpha_\phi(s)\bee{kh}]$ for every generator $u_{kh}(s)$ of $G_n$.
\end{lemma}

Lemma \ref{main4} together with Lemma \ref{main2} and Corollary \ref{main3} yields Theorem \ref{main-main}.

\section{Preliminaries} 

Assuming that the reader is familiar with the basics of cohomology theory for groups, here we recall only what is strictly necessary for the purposes of the present paper. Given a group $X$ and a (right) $X$-module $V$ defined over a given field $\KK$, a $1$-{\em cochain} of $X$ over $V$ is a mapping $f:X\rightarrow V$ such that $f(1) = 0$. A $1$-cochain $f$ is a $1$-{\em cocycle} if
\begin{equation}\label{cochain-condition}
 \forall g_1,g_2\in X: ~~  f(g_1g_2) ~ = ~ f(g_1)\cdot g_2+f(g_2).
\end{equation}
A $1$-cochain $f$ is a $1$-{\em coboundary} if there exists $v\in V$ such that $f(g) = v-v\cdot g$ for every $g\in X$. All $1$-coboundaries are $1$-cocycles.

The $1$-cochains form an abelian group (actually a $\KK$-vector space), denoted by $C^1(X,V)$. The $1$-cocycles form a subgroup (actually a vector subspace) of $C^1(X,V)$, denoted by $Z^1(X,V)$. The $1$-coboundaries form a subgroup (actually a subspace) of $Z^1(X,V)$, denoted by $B^1(X,V)$. The quotient $H^1(X,V) := Z^1(X,V)/B^1(X,V)$ is the first cohomology group of $X$ over $V$.

\subsection{A natural embedding of $Z^1(G_n, A_n)$ in $Z^1(G_n,A^*_n)$}
We now turn back to our setting, where $X = G_n$ and $V = A^*_n\cong \cM_n/\cI_n$.  As in the Introduction, $p := \cha(\KK)$. Henceforth, if $p$ is $0$ or is positive but prime with $n$ we write $p\ndiv n$; if $p$ is positive and divides $n$ then we write $p\ddiv n$.

\begin{lemma}\label{trace}
We have $Z^1(G_n, \cM_n) = Z^1(G_n, A_n)$.  If $p\ndiv n$ then $B^1(G_n, \cM_n) = B^1(G_n, A_n)$. If $p\ddiv n$ then $B^1(G_n, \cM_n) \supsetneq B^1(G_n, A_n)$.
\end{lemma}
\begin{proof}
  Let $F\in Z^1(G_n,\cM_n)$. Condition \eqref{cochain-condition}
  on $F$ implies that
  $$\Tr(F(g_1g_2)) = \Tr(F(g_1)g_2) + \Tr(F(g_2))=\Tr(F(g_1))+\Tr(F(g_2)).$$
   So, the mapping $\tau : g\in G_n \rightarrow \Tr(F(g)) \in \KK$ is a homomorphism from $G_n$ to the additive group of $\KK$. However $G_n$ is perfect. Hence $\tau$ is the trivial homomorphism: $\Tr(F(g)) = 0$ for every $g\in G_n$. The equality $Z^1(G_n, \cM_n) = Z^1(G_n, A_n)$ follows.

Obviously $B^1(G_n, A_n) \subseteq B^1(G_n, \cM_n)$. Note that if $m' = m+\lambda I_n$ then $m' - m'\cdot g = m - m\cdot g$ for every $g\in G_n$. If $p\ndiv n$ we can put $\lambda = -\Tr(m)/n$. With this choice of $\lambda$ the matrix $m'$ belongs to $A_n$. The equality $B^1(G_n, A_n) = B^1(G_n, \cM_n)$ follows. On the other hand, let $p\ddiv n$. The mapping $F_1: g\in G_n\rightarrow \bee{11}- \bee{11}\cdot g$ belongs to $B^1(G_n, \cM_n)$. We shall prove that no coboundary in $B^1(G_n, A_n)$ is equal to $F_1$.
Indeed, let $m\in \cM_n$ be any matrix such that the mapping $F_m: g\in G_n\rightarrow m-m\cdot g$ coincides with $F_1$. It is easy to check that $F_1(u_{1k}(1)) = -\bee{1k}$ for every $k > 1$. So $m - m\cdot u_{1k}(1) = -\bee{1k}$ for every $k > 1$. By similar routine computations applied to all generators $u_{h,k}(1)$ of $G_n$, we can see that this condition forces $m = tI_n - \sum_{i=2}^n\bee{ii}$ for some $t\in \KK$. Accordingly, $\Tr(m) = n(t-1)+1$. However $p\ddiv n$ by assumption. Hence $n = 0$ and $\Tr(m) = 1$. Consequently $F_m\not\in B^1(G_n,A_n)$.
\end{proof}

\begin{lemma}\label{trace bis}
Every element $f \in Z^1(G_n,A^*_n)$ admits at most one representative $F\in Z^1(G_n, A_n) = Z^1(G_n, \cM_n)$. In particular, every
$1$-coboundary $f\in B^1(G_n, A^*_n)$ admits exactly one representative
$F\in B^1(G_n, \cM_n)$.
\end{lemma}
\begin{proof}
If $F_1, F_2\in Z^1(G_n, A_n)$ represent the same element of $Z^1(G_n, A^*_n)$ then $F_1-F_2$ represents the null element of $Z^1(G_n, A^*_n)$. So, the statement of the lemma is equivalent to the following: the null element of $Z^1(G_n, A_n)$ is the unique element of $Z^1(G_n,A_n)$ which represents the null element of $Z^1(G_n,A^*_n)$. The latter claim can be proved as follows. If $F\in C^1(G_n,\cM_n)$ represents $0 \in Z^1(G_n, A^*_n)$ then there exists a mapping $\lambda:G_n\rightarrow \KK$ such that
$F(g) = \lambda(g)I_n$ for every $g\in G_n$. If moreover $F\in Z^1(G_n,\cM_n)$ ($= Z^1(G_n, A_n)$ by Lemma \ref{trace}) then condition  \eqref{cochain-condition} on $F$ implies that $\lambda$ is a homomorphism from $G_n$ to the additive group of $\KK$. However $G_n$ is perfect. Hence $\lambda$ is the trivial homomorphism. So, $F(g) = \lambda(g)I_n = 0$ for every $g\in G_n$.

The first part of the lemma is proved. Turning to the second part, let $f\in B^1(G_n, A^*_n)\subseteq Z^1(G_n, A^*_n)$. Then every representative $F$ of $f$ maps $g\in G_n$ onto $m- m\cdot g + \lambda(g)I_n$, for a given matrix $m\in \cM_n$ and a mapping $\lambda:G_n\rightarrow \KK$.
The mapping $F_0:g\in G_n \rightarrow m-m\cdot g$ still represents $f$  and is uniquely determined by the first part of the Lemma because $F_0\in B^1(G_n, \cM_n)\subseteq Z^1(G_n, \cM_n)$.
\end{proof}

Of course, every element of $Z^1(G_n, \cM_n)$ (or $B^1(G_n, \cM_n)$) represents an element of $Z^1(G_n, A^*_n)$ (respectively $B^1(G_n, A^*_n)$). Hence Lemma \ref{trace} and \ref{trace bis} immediately imply the following.

\begin{corollary}\label{trace ter}
The function which maps every $1$-cochain $F \in C^1(G_n, \cM_n)$ onto the $1$-cochain $[F]\in C^1(G_n, A^*_n)$ induces an injective homomorphism from the group $Z^1(G_n, \cM_n) = Z^1(G_n, A_n)$ into $Z^1(G_n, A^*_n)$ and isomorphically maps $B^1(G_n, \cM_n)$ onto $B^1(G_n, A^*_n)$.

In particular, if every element of $Z^1(G_n, A^*_n)$ admits a representative in the group $Z^1(G_m, \cM_n)$ then  $Z^1(G_n, A^*_n) \cong Z^1(G_n,\cM_n) =  Z^1(G_n, A_n)$ and
\[H^1(G_n, A^*_n) ~ \cong ~ H^1(G_n, \cM_n).\]
In this case, if $p\ndiv n$ then $H^1(G_n, A^*_n) \cong H^1(G_n,A_n)$, otherwise $H^1(G_n, A^*_n)$ is a proper homomorphic image of $H^1(G_n, A_n)$.
\end{corollary}

\begin{prop}\label{Trace}
  If $p\ndiv n$, then every $f\in Z^1(G_n,A^*_n)$ admits just one representative $F\in Z^1(G_n, A_n)$.
  On the other hand, when $p\ddiv n$ then, for every $f\in Z^1(G_n, A^*_n)$, all representatives of $f$ belong to $C^1(G_n, A_n)$.
\end{prop}
\begin{proof}
Let $f\in Z^1(G_n, A^*_n)$ and let $F\in C^1(G_n, \cM_n)$ be a representative of $f$. Suppose that $p\ndiv n$.
The mapping
\begin{equation}\label{trace eq}
F_0 ~:~  g\in G_n \rightarrow F(g)- \frac{\Tr(F(g))}{n}\cdot I_n
\end{equation}
is a $1$-cochain in $C^1(G_n,A_n)$ and still represents $f$. The $1$-cochain $F_0$ is the unique representative of $f$ in $Z^1(G_n, A_n)$. Indeed $F_0(g)$ is the unique traceless matrix in $f(g)$, for every $g\in G_n$. Since
\[\begin{array}{l}
\Tr(F_0(g_1)\cdot g_2 + F_0(g_2)) = \Tr(g_2^{-1}F_0(g_1)g_2 + F_0(g_2)) = \\
= \Tr(F_0(g_1)) + \Tr(F_0(g_2)) = 0+0 = 0
\end{array}\]
and $F_0(g_1g_2)$ is the unique matrix in $f(g_1g_2)\cap A_n$, necessarily $F_0(g_1)\cdot g_2 + F_0(g_2) = F_0(g_1g_2)$.
So, $F_0\in Z^1(G_n, A_n)$. By Lemma \ref{trace}, if $F \in Z^1(G_n, \cM_n)$ represents $f$ then $F\in Z^1(G_n, A_n)$. Therefore $F = F_0$, since $F_0$ is the unique representative of $f$ in $Z^1(G_n,A_n)$. The first part of the proposition is proved.

Suppose now that $p\ddiv n$. Then $\Tr(I_n) = 0$. Hence $\Tr(F(g))$ does not depend on the choice of the representative $F$ of $f$. Accordingly, with $\Tr(f(g)) := \Tr(F(g))$, condition \eqref{cochain-condition} on $f$ implies that $\Tr(f(g_1g_2)) = \Tr(f(g_1)) + \Tr(f(g_2))$. As in the proof of Lemma \ref{trace}, we obtain that $\Tr(f(g)) = 0$ for every $g\in G_n$.
\end{proof}

In view of Proposition \ref{Trace}, when $p\ndiv n$, given a $1$-cocycle $f\in Z^1(G_n,A^*_n)$ we know how to choose a $1$-cocycle in $Z^1(G_n,A_n)$ as the best representative of $f$ (formula \eqref{trace eq} tells us how to do). In this case $Z^1(G_n, A^*_n) \cong Z^1(G_n, A_n)$ and $H^1(G_n, A^*_n) \cong H^1(G_n, A_n)$ by Corollary \ref{trace ter}, as we already know from Lemma \ref{pchar}.

In contrast, when $p \ddiv n$ there is no easy way to make that choice. In this case, the equalities which hold for $f$ hold up to scalar matrices when we replace $f$ with any of its representatives. This can be annoying. However, in many cases the following obvious remark allows us to get rid of uncomfortable scalar matrices:

\begin{itemize}
\item[(CA)] ({\bf Cleaning Argument}). Given a matrix $x =\sum_{i,j}x_{ij}\bee{ij}$, suppose that $x_{ii} = 0$ for some $i \in \{1, 2,\dots, n\}$. Then $x\in \cI_n$ only if $x = 0$.
\end{itemize}

\subsection{Some information on the generators $u_{ij}(t)$ of $G_n$}

Recall that $G_n$ is generated by the elements $u_{ij}(t) = I_n+t\bee{ij}$ ($i\neq j$ and $t\in \KK$) where $\bee{ij}$ is the $n\times n$ matrix having $1$ in position $(i,j)$ and $0$ in all the other entries. In this subsection we shall list a number of properties of the generators $u_{ij}(s)$ of $G_n$. All properties we are going to mention are well known. All of them can be obtained by repeated applications of the following basic property:
\[ \bee{ij}\bee{hk}=\begin{cases}
   \bee{ik} & \text{ if $j=h$ } \\
   0     & \text{ if $j\neq h$}.
 \end{cases} \]
We leave these deductions to the reader.
  \begin{enumerate}[{\rm (U1)}]
\item\label{U1} $u_{ij}(s)u_{ij}(t) = u_{ij}(s+t)$; hence $u_{ij}(t)^{-1} = u_{ij}(-t)$, since $u_{ij}(0) = 1$.
\item\label{U2}  $u_{ij}(t)u_{hk}(s) = \begin{cases}
I_n+t\bee{ij}+s\bee{jk}+ts\bee{ik} & \text{ if } j = h,\\
I_n+t\bee{ij}+s\bee{hk} & \text{ if } j \neq h.
\end{cases}$
\item\label{U3} $u_{pq}(s)u_{kh}(t)=u_{kh}(t)u_{pq}(s)$ if and only if  $q\neq k$
  and $h\neq p$.
\item\label{U4} If $k \neq q$ then $u_{kh}(s)u_{hq}(t)u_{kh}(s)^{-1}u_{hq}(t)^{-1} = u_{k,q}(st)$.
\item\label{U5} For $k\neq h$ and $t\neq 0$, put $w_{kh}(t):=u_{kh}(t)u_{hk}(-t^{-1})u_{kh}(t)$.

Then $w_{kh}(t) = I_n - (\bee{k,k}+\bee{h,h})+ (t\bee{k,h}-t^{-1}\bee{h,k})$.
\item\label{U6} $w_{kh}(t)^{-1}=w_{kh}(-t)$.
\item\label{U7} For $k\neq h$ and $t\neq 0$ put $h_{kh}(t):=w_{kh}(t)w_{kh}(1)^{-1}$.

Then $h_{kh}(t) = I_n+(t-1)\bee{kk}+(t^{-1}-1)\bee{hh}$.
\item\label{U8} $h_{kh}(t)h_{kh}(s) = h_{kh}(ts)$.
\item\label{U9} $w_{kh}(t)w_{kh}(s) = h_{kh}(-t/s)$.
\end{enumerate}

\section{Proof of Lemma \ref{main2}}\label{TZ} 

The mapping $F_{\alpha}: g\in G_n\rightarrow g^{-1}{\alpha}(g)\in \cM_n$ is a representative of the $1$-cochain $f_{\alpha}\in C^1(G_n, A^*_n)$ considered in Lemma \ref{main2}. It is straightforward to check that $F_{\alpha}$ satisfies condition \eqref{cochain-condition}. Hence $F_{\alpha}$ is a $1$-cocycle (in $Z^1(G_n, A_n)$, by Lemma~\ref{trace}). Accordingly, $f_{\alpha}\in Z^1(G_n, A^*_n)$. It remains to prove that $f_{\alpha}\in B^1(G_n, A^*_n)$ only if ${\alpha} = 0$.

Suppose that $f_{\alpha}\in B^1(G_n, A^*_n)$. Then there exists a matrix $m\in \cM_n$, say $m = \sum_{i,j=1}^n\bee{ij}m_{ij}$, and a mapping $\lambda: G_n\rightarrow \KK$ such that $g^{-1}{\alpha}(g) =  m - g^{-1}mg + \lambda(g)I_n$ for every $g\in G_n$. In particular, with $g = u_{kh}(s)$,
\[(I_n-s\bee{kh}){\alpha}(s)\bee{kh} ~= ~ s\bee{kh}m - sm\bee{kh} +s^2\bee{kh}m\bee{kh} + \lambda_{k,h;s}I_n\]
where we put $\lambda_{k,h;s} := \lambda(u_{kh}(s))$ for short. However $(I_n-s\bee{kh}){\alpha}(s)\bee{kh} = {\alpha}(s)\bee{kh}$ while
\[\begin{array}{ll}
s\bee{kh}m - sm\bee{kh} +s^2\bee{kh}m\bee{kh} + \lambda_{k,h;s}I_n = \\
= s(\sum_{j=1}^nm_{hj}\bee{kj} - \sum_{i=1}^nm_{ik}\bee{ih} + sm_{hk}\bee{kh}) + \lambda_{k,h,s}I_n = \\
= s(\sum_{j\neq h, k}m_{hj}\bee{kj} - \sum_{i\neq k, h}m_{ik}\bee{ih} + (sm_{hk} + m_{hh}- m_{kk})\bee{kh} + \\
~~~ + m_{hk}(\bee{kk}- \bee{hh})) + \lambda_{k,h;s}I_n.
\end{array}\]
Therefore
\begin{equation}\label{lunga}
\begin{array}{rcl}
{\alpha}(s)\bee{kh} & = & s\sum_{j\neq h, k}m_{hj}\bee{kj} - s\sum_{i\neq k, h}m_{ik}\bee{ih} +\\
 & & + s(sm_{hk} + m_{hh}- m_{kk})\bee{kh} + sm_{hk}(\bee{kk}- \bee{hh}) + \\
 & & +\lambda_{k,h;s}I_n.
\end{array}
\end{equation}
In particular, with $s = 1$,
\[\begin{array}{rcl}
0 & = & \sum_{j\neq h, k}m_{hj}\bee{kj} - \sum_{i\neq k, h}m_{ik}\bee{ih} +\\
 & & + (m_{hk} + m_{hh}- m_{kk})\bee{kh} + m_{hk}(\bee{kk}- \bee{hh}) + \lambda_{k,h;1}I_n.
\end{array}\]
It follows that $m_{hj} = 0$ for every $j\neq h, k$, $m_{ik} = 0$ for every $i \neq k, h$, $m_{hk} = m_{kk}- m_{hh}$, $\lambda_{k,h;1} = 0$ (by the Cleaning Argument) and, ultimately, $m_{h,k} =0$. Hence $m_{kk} = m_{hh}$. However this holds for every choice of $k$ and $h$ with $k\neq h$. Consequently, $m\in \cI_n$. Equation \eqref{lunga} thus amounts to ${\alpha}(s)\bee{kh} = \lambda_{k,h;s}I_n$. We have $\lambda_{k,h;s} = 0$ by the Cleaning Argument. Therefore ${\alpha}(s) = 0$ for any $s \in \KK$. In short, ${\alpha} = 0$.  The proof of Lemma \ref{main2} is complete. {\hfill$\square$\par}

  \section{Proof of Lemma~\ref{main4}} 
Lemma~\ref{main4} will be proved in a series of steps. In order to keep track of the several intermediate results, we summarize them here schematically.

We first distinguish the case $n>4$ and $\KK$ arbitrary or $n = 4$ and $\KK\not=\FF_2$ from the case $n = 4$ and $\KK=\FF_2$.
The general case will be treated in Subsection~\ref{caso n>3} while the case $(n,\KK)=(4,\FF_2)$ will be done in  Subsection~\ref{caso n=4 e K=F_2}.

In Subsection~\ref{caso n>3} we will start proving that for every $f\in Z^1(A_n^*, G_n)$ it is always possible to choose a representative $F\in C^1(A_n,G_n)$ such that, for any generator $u_{kh}(t)$ of the group $G_n$, all entries of $F(u_{kh}(t))$ are null but possibly the $(k,h)$-entry. are in row $k$. This will be done in Lemmas~\ref{s1l}, \ref{thZ1} and~\ref{rn>3}. 

With $F$ as above, from Lemma~\ref{e7} to Lemma~\ref{uni}  we prove that  $F(u_{kh}(t))=\alpha(t)\bee{kh}$, where $\alpha\colon \KK\rightarrow \KK$ is a derivation of $\KK$ not depending on $(k,h)$, as claimed in Lemma \ref{main4}.

\subsection{General case: $n>3$ and $(n, \KK) \neq (4, \FF_2)$}\label{caso n>3}

Throughout this section we assume $n>4$ and $\KK$ arbitrary or $n = 4$ and $\KK\not=\FF_2$. Given  $f\in Z^1(G_n,A_n^*)$, let $F$ be a  representative of $f$ in $C^1(G_n,A_n)$. Such an $F$ exists by Proposition~\ref{Trace}. Moreover, if $p\ndiv  n$ then $F$ is unique and belongs to $Z^1(G_n,A_n)$. In contrast, when $p\ddiv n$ we know only that $F\in C^1(G_n,A_n)$. 

In order to discuss both cases together we introduce the symbol $\equiv$, to be read as follows: if $p\ndiv n$ then $\equiv$ stands for $=$ while when $p\ndiv n$ then the symbol $\equiv$ means equivalence modulo $\cI_n$.   

Given a generator $u_{kh}(t)$ of $G_n$ we put
\[\sum_{i,j=1}^n a_{ij}^{kh}(t)\bee{ij} := F(u_{kh}(t)) \]
The following three lemmas give information on the structure of the matrix $F(u_{kh}(t)) = \sum_{i,j=1}^n a_{ij}^{kh}(t)\bee{ij}.$ 

\begin{lemma} \label{s1l}
We have
\begin{equation}\label{d}
F(u_{kh}(t)) ~ \equiv ~ \sum_{i\neq h}a^{kh}_{ki}(t)\bee{ki}+
    \sum_{j\neq k}a^{kh}_{jh}(t)\bee{jh}+a^{kh}_{kh}(t)\bee{kh}+a^{kh}_{hk}(t)\bee{hk}.
\end{equation}
\end{lemma}
\begin{proof}
  Take $p,q$ such that $1\leq p,q\leq n$ and
  $|\{p,q,k,h\}|=4$. Then, by (U\ref{U3}),  we have $u_{pq}(s)u_{kh}(t)=u_{kh}(t)u_{pq}(s)$,
so
  \begin{equation}
    \label{cco}
    f(u_{pq}(s)u_{kh}(t))=f(u_{kh}(t)u_{pq}(s)).
  \end{equation}
  By the cochain condition~\eqref{cochain-condition}
  we have
    \begin{multline*}
      F(u_{kh}(t)u_{pq}(s))\equiv u_{pq}^{-1}(s)F(u_{kh}(t))u_{pq}(s)+F(u_{pq}(s)) = \\
     =  u_{pq}(-s)F(u_{kh}(t))u_{pq}(s)+F(u_{pq}(s)) = \\
     =  (I_n-s \bee{pq})F(u_{kh}(t))(I_n+s \bee{pq})+F(u_{pq}(s)) = \\
    =  F(u_{kh}(t))-s \bee{pq}F(u_{kh}(t))+
      s F(u_{kh}(t)) \bee{pq}-s^2\bee{pq}F(u_{kh}(t))\bee{pq}+F(u_{pq}(s)).
  \end{multline*}
\begin{multline*} F(u_{kh}(t)u_{pq}(s)) \equiv \\
\equiv (F(u_{kh}(t))+F(u_{pq}(s))) + s ( F(u_{kh}(t))\bee{pq}-\bee{pq}F(u_{kh}(t))-s \bee{pq}F(u_{kh}(t))\bee{pq}).
\end{multline*}
 Analogously,
  \begin{multline*}
    F(u_{pq}(s)u_{kh}(t))\equiv u_{kh}^{-1}(t)F(u_{pq}(s))u_{kh}(t)+F(u_{kh}(t)) = \\
    = u_{kh}(-t)F(u_{pq}(s))u_{kh}(t)+F(u_{kh}(t)) = \\
   =  (I_n-t \bee{kh})F(u_{pq}(s))(I_n+t \bee{kh})+F(u_{kh}(t))  = \\
  =   F(u_{pq}(s))-t \bee{kh}F(u_{pq}(s))+t F(u_{pq}(s)) \bee{kh} -t^2\bee{kh}F(u_{pq}(s))\bee{kh}+F(u_{kh}(t));
    \end{multline*}
hence
\begin{multline*}
	    F(u_{pq}(s)u_{kh}(t)) \equiv\\
            \equiv ( F(u_{pq}(s))+F(u_{kh}(t)))+ t (F(u_{pq}(s))\bee{kh}-\bee{kh}F(u_{pq}(s))-t \bee{kh}F(u_{pq}(s))\bee{kh}).
  \end{multline*}
However $F(u_{pq}(s)u_{kh}(t))\equiv F(u_{kh}(t)u_{pq}(s))$ by \eqref{cco}. This and above imply the following:
    \begin{multline}
    \label{c}
    s(F(u_{kh}(t))\bee{pq}-\bee{pq}F(u_{kh}(t))-s \bee{pq}F(u_{kh}(t))\bee{pq}) \equiv \\
   \equiv t(F(u_{pq}(s))\bee{kh}-\bee{kh}F(u_{pq}(s))-t \bee{kh}F(u_{pq}(s))\bee{kh}).
  \end{multline}
  Define now $A(s,t)$ and $B(s,t)$ as the matrices appearing on
  respectively the left and the right hands of Equation~\eqref{c}, i.e.
 \[\begin{array}{rcl}
A(s,t) & := & s( {{F(u_{kh}(t))\bee{pq}}}-{\bee{pq}F(u_{kh}(t))}-s \bee{pq}F(u_{kh}(t))\bee{pq}),\\
B(s,t) & := & t({{F(u_{pq}(s))\bee{kh}}}-{\bee{kh}F(u_{pq}(s))}-t \bee{kh}F(u_{pq}(s))\bee{kh}). 
\end{array}\]
 Recalling that $F(u_{kh}(t)):=\sum_{i,j=1}^na^{kh}_{ij}(t)\bee{ij}$, we have
 \begin{itemize}
  \item
    $F(u_{kh}(t))\bee{pq}=\sum_{i=1}^n \bee{iq}a^{kh}_{ip}(t)$
  \item
    $\bee{pq} F(u_{kh}(t))=\sum_{i=1}^n \bee{pi}a^{kh}_{qi}(t)$
  \item
    $\bee{pq} F(u_{kh}(t)) \bee{pq}= \bee{pq} a^{kh}_{qp}(t)$
  \end{itemize}
  and, analogously
  \begin{itemize}
 \item
    $F(u_{pq}(s))\bee{kh}=\sum_{i=1}^n \bee{ih}a^{pq}_{ik}(s)$
  \item
    $\bee{kh} F(u_{pq}(s))=\sum_{i=1}^n \bee{ki}a^{pq}_{hi}(s)$
  \item
    $\bee{kh} F(u_{pq}(s)) \bee{kh}= \bee{kh} a^{pq}_{hk}(s)$.
  \end{itemize}
So,
  \begin{multline*} A(s,t) =
    s(-{\sum_{i=1}^n \bee{pi}a^{kh}_{qi}(t)} + { \sum_{i=1}^n \bee{iq}a^{kh}_{ip}(t)} - s\bee{pq} a^{kh}_{q,p}(t))=\\
    -s({\sum_{i\neq q,p,h} \bee{pi}a^{kh}_{qi}(t)+\bee{pq}a^{kh}_{qq}(t)+\bee{pp}a^{kh}_{qp}(t)+\bee{ph}a^{kh}_{qh}(t)})+\\
    s({ \sum_{i\neq p,k,q} \bee{iq} a^{kh}_{ip}(t)+ \bee{pq}a^{kh}_{pp}(t)+
        \bee{k,q}a^{kh}_{kp}(t)+\bee{qq}a^{kh}_{qp}(t)})-s^2\bee{pq}a^{kh}_{qp}(t)= \\
    s\Big({-\sum_{i\neq q,p,h}\bee{pi}a^{kh}_{qi}(t)+\bee{pp}a^{kh}_{qp}(t)-\bee{ph}a^{kh}_{qh}(t)}
    \\
    +{ \sum_{i\neq p,k,q} \bee{iq} a^{kh}_{ip}(t)+
      \bee{kq}a^{kh}_{kp}(t)+\bee{q,q}a^{kh}_{qp}(t)}\
    +({ a^{kh}_{pp}(t)}-{ a^{kh}_{qq}(t)}-sa^{kh}_{qp}(t))\bee{pq}\Big).
  \end{multline*}
  and
    \begin{multline*}
    B(s,t)=
    t  {{F(u_{pq}(s))\bee{kh}}}-{\bee{kh}F(u_{pq}(s))}-t \bee{kh}F(u_{pq}(s))\bee{kh}= \\
    t -{\sum_{i=1}^n \bee{ki}a^{pq}_{hi}(s)} + { \sum_{i=1}^n \bee{ih}a^{pq}_{ik}(s)} - t\bee{kh} a^{pq}_{hk}(s)=\\
    -t {\sum_{i\neq h,k,q} \bee{ki}a^{pq}_{hi}(s)+\bee{kh}a^{pq}_{hh}(s)+\bee{kk}a^{pq}_{hk}(s)+\bee{kq}a^{pq}_{hq}(s)} +\\
    t{ \sum_{i\neq k,p,h} \bee{ih} a^{pq}_{ik}(s)+ \bee{kh}a^{pq}_{kk}(s)+
        \bee{ph}a^{pq}_{pk}(s)+\bee{hh}a^{pq}_{hk}(s)}-t^2\bee{kh}a^{pq}_{hk}(s)= \\
    t\Big( {-\sum_{i\neq h,k,q}\bee{ki}a^{pq}_{hi}(s)+\bee{kk}a^{pq}_{hk}(s)-\bee{kq}a^{pq}_{hq}(s)}
    \\
    +{ \sum_{i\neq k,p,h} \bee{ih} a^{pq}_{ik}(s)+
      \bee{ph}a^{pq}_{pk}(s)+\bee{hh}a^{pq}_{hk}(s)}\
    +({ a^{pq}_{kk}(s)}-{ a^{pq}_{hh}(s)}-ta^{pq}_{hk}(s))\bee{kh}\Big).
  \end{multline*}
Equation ~\eqref{c} can be rewritten as $A(s,t)\equiv B(s,t)$,  that is
  \[ A(s,t) = B(s,t)+\lambda(s,t)I_n \]
  for some scalar function $\lambda(s,t):\KK\times\KK\to\KK$ (which, when $p\ndiv n$ is the null function). 
  Comparing the expressions of $A(s,t)$ and $B(s,t)$ we get the conditions
  T\ref{tab-c}.1--T\ref{tab-c}.12
  of Table~\ref{tab-c} on the coefficients of $F(u_{kh}(t))$ and $F(u_{pq}(s)).$ 

  \begin{table}
    \newcounter{condNo}
    \newcommand{\itm}{\stepcounter{condNo}\rm{T}\ref{tab-c}.\arabic{condNo}}
 \[   \setlength\extrarowheight{2pt}
\begin{array}{l|l|l|l}
& \text{element}& \text{condition on coefficient} & \text{assumptions} \\
\hline
\itm& \bee{pi} & a^{kh}_{qi}(t)=0 & i\not=q,p,h \\
 \hline
\itm & \bee{pp} & sa^{kh}_{qp}(t)=\lambda(s,t)& \forall s\\
\hline
\itm& \bee{ph} & sa^{kh}_{qh}(t)+t a^{pq}_{pk}(s)=0 & \forall s,t \\
\hline
\itm& \bee{iq} & sa^{kh}_{ip}(t)=0 & i\not=p,k,q \\
\hline
\itm& \bee{kq} & sa^{kh}_{kp}(t)+t a^{pq}_{hq}(s)=0 & \forall s,t\\
\hline
\itm& \bee{qq} & sa^{kh}_{qp}(t)=\lambda(s,t) & \forall s,t \\
\hline
\itm& \bee{pq} & sa^{kh}_{pp}(t)-sa^{kh}_{qq}(t)-s^2a^{kh}_{qp}(t)=0 & \forall s,t  \\
\hline
\itm& \bee{ki} & a^{pq}_{hi}(s)=0 & i\not=h,k,q\\
 \hline
\itm& \bee{kk} & ta^{pq}_{hk}(s)=\lambda(s,t)&\forall s,t \\
\hline
\itm&
     \bee{ih} & a^{pq}_{ik}(s)=0 & i\not=k,p,h\\
 \hline
\itm& \bee{hh} & -ta^{pq}_{hk}(s)=\lambda(s,t)&\forall s,t  \\
\hline
\itm& \bee{kh} & -ta^{pq}_{kk}(s)+ta^{pq}_{hh}(s)+t^2a^{kh}_{hk}(s)=0 & \forall s,t  \\
 \end{array}
\]
\caption{Conditions on $F(u_{kh}(t))$ and $F(u_{pq}(s))$ with $|\{k,h,p,q\}| = 4$.}
\label{tab-c}
\end{table}

If $\mathrm{char}(\KK)\neq2$, conditions T\ref{tab-c}.9 and T\ref{tab-c}.11 together directly imply $\lambda(s,t)=0$ for all $s,t$.
So, condition T\ref{tab-c}.2 gives $a_{qp}^{kh}(t)=0$ for all $t$ and
for all $q,p\not\in\{k,h\}$, $p\neq q$. Also
condition \ref{tab-c}.7 yields $a_{pp}^{kk}(t)=a_{qq}^{kk}(t)$ for all
$p,q\not\in\{k,h\}$.

We now consider the case $\mathrm{char}(\KK)=2$.
By T\ref{tab-c}.2 and T\ref{tab-c}.11, we have
$\lambda(s,t)=ta_{hk}^{pq}(s)=sa_{qp}^{kh}(t)$.
So, if we put $t=1$ we have
$a_{hk}^{pq}(s)=sa_{qp}^{kh}(1)$ whence
$ta_{hk}^{pq}(s)=st a_{qp}^{kh}(1)=sa_{qp}^{kh}(t)$
and, dividing by $s$,
$a_{qp}^{kh}(t)=t a_{qp}^{kh}(1)$.
Likewise $a_{hk}^{pq}(s)=s a_{qp}^{kh}(1)$. It follows by
\ref{tab-c}.9,
$\lambda(s,t)=ta_{hk}^{pq}(s)=(st) a_{hk}^{pq}(1)=(st) a_{qp}^{kh}(1)$. There are now two possibilities:
\begin{enumerate}[a)]
	\item
	$|\KK|>2$, $n\geq 4$.
	By~T\ref{tab-c}.7 we have
	$sa_{qp}^{kh}(t)=a_{pp}^{kh}(t)+a_{qq}^{kh}(t)$ for at least $2$ distinct
	values of $s\neq 0$. So $a_{qp}^{kh}(t)=0$ and $a_{pp}^{kh}(t)=a_{qq}^{kh}(t)$.
	The former yields by~T\ref{tab-c}.2 $\lambda(s,t)=0$.
	By applying the same argument for all $p\neq q$ with $p,q\neq k,h$ we
	get also $a_{pp}^{kh}(t)=a_{qq}^{kh}(t)$ for all $p,q\not\in\{h,k\}$.
	
	\item $\KK=\FF_2$ and $n>4$.
	If $n>4$ and $\KK=\FF_2$, then by considering three elements $u_{kh}(t)$,
	$u_{pq}(s)$ and $u_{px}(r)$ with $|\{k,h,p,q,x\}|=5$ by T\ref{tab-c}.2 we get
	$\lambda(s,t)=sa_{qp}^{kh}(t)$.
	On the other hand, by T\ref{tab-c}.4 applied to $u_{kh}(t)$ and
	$u_{px}(r)$ we have $ra_{qp}^{kh}(t)=0$ for $r=1$.
	This gives $\lambda(s,t)=0$. Now, by the same condition we also have
	$a_{pp}^{kh}(t)=a_{qq}^{kh}(t)$ for all $p,q\not\in\{h,k\}$.
\end{enumerate}
As $p,q$ vary in all possible ways, in all of the cases being
considered we have $\lambda(s,t)=0$ and
$a^{kh}_{pp}(t)=a^{kh}_{qq}(t)$. So, there exists a function
$\rho:\KK\to\KK$ such that
\begin{equation}
	\label{eZ0}
	\begin{cases}
		a^{kh}_{pp}(t)=\rho(t) & \forall p\not\in\{h,k\} \\
		a^{kh}_{pq}(t)=0 & \forall p\neq q, p,q\not\in\{h,k\}.
	\end{cases}
\end{equation}
Conditions~\eqref{eZ0} apply to all matrices $F(u_{kh}(t))\in f(u_{kh}(t))$.
On the other hand, since $f(u_{kh}(t))\in\cM_n/\cI_n$, we can
always choose as canonical representative of $f(u_{kh}(t))$ the unique matrix
therein contained for which $\rho(t)=0$, i.e. for which $a^{kh}_{ij}(t)=0$ if
$\{i,j\}\cap\{h,k\}=\emptyset$.
This gives the lemma.
\end{proof}

When $p\ddiv n$ we can choose $F$ in such a way that 
\[F(u_{kh}(t)) ~ = ~ \sum_{i\neq h}a^{kh}_{ki}(t)\bee{ki}+
    \sum_{j\neq k}a^{kh}_{jh}(t)\bee{jh}+a^{kh}_{kh}(t)\bee{kh}+a^{kh}_{hk}(t)\bee{hk}.\]
Indeed when $p\ddiv n$ we can safely modify the values of $F$ by adding scalar matrices. We can always choose such a matrix in such a way that the equivalence \eqref{d} becomes an equality. Accordingly, henceforth we assume that \eqref{d} is indeed an equality. 

\begin{lemma}
  \label{thZ1}
We have
  \begin{equation}\label{d1}
    F(u_{kh}(t))=
    \sum_{i\neq h,k}a^{kh}_{ki}(t)\bee{ki}+\sum_{j\neq h,k}a^{kh}_{jh}(t)\bee{jh}+
    a^{kh}_{kh}(t)\bee{kh}+ta^{kh}_{hh}(1)(\bee{hh}-\bee{kk}).
  \end{equation}
  Furthermore, for all $p,q$ such that $|\{h,k,p,q\}|=4$ we have
  \begin{equation}
    \begin{cases}
    sa_{qh}^{kh}(t)=-ta_{kk}^{kq}(s) \\
    sa_{kk}^{kh}(t)=-ta_{hq}^{kq}(s) \\
    sa_{kp}^{kh}(t)=-ta_{hh}^{ph}(s) \\
    sa_{hh}^{kh}(t)=-ta_{ph}^{ph}(s)
  \end{cases}
  \end{equation}
\end{lemma}\begin{proof}
  First observe that $\Tr(F(u_{kh}(t))) = 0$ (by choice of $F$ when $p\ndiv n$ and by Proposition \ref{Trace} when $p\ddiv n$. So, by Lemma~\ref{s1l}
  we have $a^{kh}_{kk}(t)=-a^{kh}_{hh}(t)$.
  Take now $q\not\in\{k,h\}$. Then,
  since  $u_{kh}(t)u_{kq}(s)=u_{kq}(s)u_{kh}(t)$ by (U\ref{U3}), we also have
$F(u_{kh}(t)u_{kq}(s))=F(u_{kq}(s)u_{kh}(t)).$
  Then,
\begin{multline*}
  F(u_{kh}(t)u_{kq}(s))\equiv
  (F(u_{kh}(t))+F(u_{kq}(s)))+\\
  +sF(u_{kh}(t)\bee{kq}-s\bee{kq}F(u_{kh}(t))-s^2\bee{kq}F(u_{kh}(t))\bee{kq}
  = \\
  \\
    (F(u_{kh}(t))+F(u_{kq}(s)))+ \\
  s\left(\sum_{i\neq h} a^{kh}_{ki}(t)\bee{ki}+
    \sum_{j\neq k}a^{kh}_{jh}(t)\bee{jh}+a^{kh}_{kh}(t)\bee{kh}+
    a^{kh}_{hk}(t)\bee{hk}\right)\bee{kq}+\\
  -s\bee{kq}\left(\sum_{i\neq h} a^{kh}_{ki}(t)\bee{ki}+
    \sum_{j\neq k}a^{kh}_{jh}(t)\bee{jh}+a^{kh}_{kh}(t)\bee{kh}+
    a^{kh}_{hk}(t)\bee{hk}\right)+\\
  -s^2\bee{kq}\left(\sum_{i\neq h} a^{kh}_{ki}(t)\bee{ki}+
    \sum_{j\neq k}a^{kh}_{jh}(t)\bee{jh}+a^{kh}_{kh}(t)\bee{kh}+
    a^{kh}_{hk}(t)\bee{hk}\right)\bee{kq} =\\
   (F(u_{kh}(t))+F(u_{kq}(s)))+
  s(a^{kh}_{kk}(t)\bee{kq}+a^{kh}_{hk}(t)\bee{hq}
  -a^{kh}_{qh}(t)\bee{kh}).
\end{multline*}
So,
\[ F(u_{kh}(t)u_{kq}(s))\equiv  (F(u_{kh}(t))+F(u_{kq}(s)))+
  s(a^{kh}_{kk}(t)\bee{kq}+a^{kh}_{hk}(t)\bee{hq}
  -a^{kh}_{qh}(t)\bee{kh}).\]
Similarly,
\[ F(u_{kq}(s)u_{kh}(t)) \equiv 
  (F(u_{kh}(t))+F(u_{kq}(s)))+
  t(a^{kq}_{kk}(s)\bee{kh}+a^{kq}_{qk}(s)\bee{qh}-a^{kq}_{hq}(s)\bee{kq}).\]
The equality $F(u_{kh}(t)u_{kq}(s)) =  F(u_{kq}(s)u_{kh}(t))$ and the previous two equivalences yield $t\left(a^{kq}_{kk}(s)\bee{kh}+a^{kq}_{qk}(s)\bee{qh}-a^{kq}_{hq}(s)\bee{kq}\right) \equiv 
  s\left(a^{kh}_{kk}(t)\bee{kq}+a^{kh}_{hk}(t)\bee{hq}
  -a^{kh}_{qh}(t)\bee{kh}\right)$, whence
\begin{equation}\label{ahk}
\begin{cases}
  a^{kh}_{hk}(t)=0=a^{kq}_{qk}(s)\\
  s a^{kh}_{qh}(t)=-ta^{kq}_{kk}(s) \\
  s a^{kh}_{kk}(t)=-ta^{kq}_{hq}(s).
\end{cases}
\end{equation}
With $s=1$ we get $a^{kh}_{kk}(t)=-ta^{kq}_{hq}(1)$.
With $s=t=1$, we get $a^{kh}_{kk}(1)=-a^{kq}_{hq}(1)$, whence
\begin{equation}
  \label{akk}
  a^{kh}_{kk}(t)=ta^{kh}_{kk}(1).
\end{equation}
Arguing in a similar way with the equality $u_{kh}(t)u_{ph}(s)=u_{ph}(s)u_{kh}(t)$ and using~\eqref{ahk} we get
\[ F(u_{kh}(t)u_{ph}(s)) \equiv 
  (F(u_{kh}(t))+F(u_{ph}(s)))+
  s(a^{kh}_{kp}(t)\bee{kh}-a^{kh}_{hh}(t)\bee{ph})\]
and
\[F(u_{ph}(s)u_{kh}(t)) \equiv
  (F(u_{kh}(t))+F(u_{ph}(s)))+
  t(a^{ph}_{pk}(s)\bee{ph}-a^{ph}_{hh}(s)\bee{kh}-a^{ph}_{hp}(s)\bee{kp}).\]
So, $s[a^{kh}_{kp}(t)\bee{kh}-a^{kh}_{hh}(t)\bee{ph}] \equiv t[a^{ph}_{pk}(s)\bee{ph}-a^{ph}_{hh}(s)\bee{kh}- a^{ph}_{hp}(s)\bee{kp}]$, 
whence
\begin{equation}
  \label{akp}
  \begin{cases}
    sa^{kh}_{kp}(t)=-ta^{ph}_{hh}(s) \\
    -sa^{kh}_{hh}(t)=ta^{ph}_{pk}(s)\\
    a^{ph}_{hp}(s)=0
  \end{cases}
\end{equation}
With $s=1$ we get $a^{kh}_{hh}(t)=-ta^{ph}_{pk}(1)$ and subsequently
with $t=1$ we get $a^{kh}_{hh}(1)=-a^{ph}_{pk}(1)$, whence
\begin{equation}
  \label{ahh}
  a^{kh}_{hh}(t)=ta^{kh}_{hh}(1).
\end{equation}
Implementing \eqref{akk}, \eqref{ahh} and \eqref{ahk} in ~\eqref{d}, the lemma follows.
\end{proof}

\begin{lemma}
  \label{rn>3}
Let $\phi\in H^1(G_n,A_n^*)$. Then $\phi$ admits a representative $f_\phi\in Z^1(G_n, A^*_n)$ such that for every choice of distinct indexes $k, h \in \{1, 2,...., n\}$ there exist a function
  $\eta_{kh}:\KK\to\KK$ such that 
  \begin{equation}
    \label{dM+}
    f_\phi(u_{kh}(t))=[\eta_{kh}(t)\bee{kh}], ~~~ \forall t\in\KK.  
  \end{equation}
\end{lemma}
\begin{proof}
    For $[M]\in A_n^*$ represented by a matrix
    $M=\sum_{i,j=1}^nm_{ij}\bee{ij}$,
  denote by $w_M:G_n\to A_n^*$ the function  
  given by
	$w_M(g):=[M-g^{-1}Mg]$.
  Observe that if $N\in [M]$ then $N=M+\alpha I_n$ and
  $w_N(g)=[N-g^{-1}Ng]=[M+\alpha I_n-g^{-1}(M+\alpha I_n)g]=
  [M-g^{-1}Mg]$; so $w_M$ depends only on $[M]$ and is
  the $1$-coboundary defined by $[M]$.
  Then
  {\allowdisplaybreaks
  \begin{multline}
  \label{dM}
     w_M(u_{kh}(t))=
    [ M-(I_n-t\bee{kh})M(I_n+t\bee{kh}) ]=\\
  =  [ M-M+t\bee{kh}M-tM\bee{kh}+t^2\bee{kh}M\bee{kh} ]= \\
    \left[ t\sum_{i,j=1}^nm_{ij}\bee{kh}\bee{ij}-
    t\sum_{i,j=1}^nm_{ij}\bee{ij}\bee{kh}+t^2\sum_{i,j=1}^n\bee{kh}\bee{ij}\bee{kh}
    \right]
    =\\ \allowbreak
   \left[ t\sum_{j=1}^nm_{hj}\bee{kj}-t\sum_{i=1}^nm_{ik}\bee{ih}+t^2m_{hk}\bee{kh}
    \right] =\\
     \left[ t\left(\sum_{\begin{subarray}{c}j=1\\
           j\neq h\end{subarray}}^nm_{hj}\bee{kj}-
       \sum_{\begin{subarray}{c}i=1\\
           i\neq k\end{subarray}}^nm_{ik}\bee{ih}\right)
     +t(tm_{hk}+m_{hh}-m_{kk})\bee{kh} \right]=\\  \allowbreak
     \left[ t\sum_{\begin{subarray}{c}j=1\\
           j\neq h,k\end{subarray}}^n(m_{hj}\bee{kj}-m_{jk}\bee{jh}) 
     +t(tm_{hk}+m_{hh}-m_{kk})\bee{kh}+tm_{hk}(\bee{kk}-\bee{hh})
     \right]. 
\end{multline}}
Let $F\in C^1(G_n,A_n)$ be such that $f_\phi=[F]$. In particular $F$ is
   a cochain representing $\phi$, i.e. $\phi=\llbracket F\rrbracket$.
   As in~\eqref{d1}, denote by $a_{hh}^{ji}(1)$ the entry in position $h,h$
   of the matrix $F(u_{ji}(1))\in A_n$.

   Choose now $M$ (equivalently, the $1$-coboundary defined by $M$) in such a way that the matrix $M$ has the following structure: $m_{ii}=0$ for all $i=1,\dots,n$
   and $m_{ij}=a_{hh}^{ji}(1)$.
      In other words
   \[ M=\sum_{\begin{subarray}{c}
         h,k=1\\
         h\neq k
       \end{subarray}
       }^n a_{hh}^{hk}(1)\bee{kh}. \]
     Clearly $\llbracket f \rrbracket=\llbracket f+w_M\rrbracket
     =\llbracket [F]+w_m\rrbracket$.
   So,
   \begin{multline}\label{e18}
      ([F]+w_M)(u_{kh}(t))=\\
     \left[\sum_{i\neq h,k}(a^{kh}_{ki}(t)+tm_{hi})\bee{ki}+
     \sum_{j\neq h,k}(a^{kh}_{jh}(t)-tm_{jk})\bee{jh}+(a_{kh}^{kh}(t)+t^2m_{hk})\bee{kh}
     \right].
    \end{multline}
    If we put $\eta^{kh}_{ki}(t):=a^{kh}_{ki}(t)+tm_{hi}$,
    $\eta^{kh}_{jh}(t):=a^{kh}_{jh}(t)-tm_{jk}$ and
    $\eta^{kh}_{hh}(t):=a_{kh}^{kh}(t)+t^2m_{hk}$ we can rewrite~\eqref{e18}
    as
  \begin{multline}
    \label{dM-}
    (f+w_M)(u_{kh}(t))=([F]+w_M)(u_{kh}(t))=\\
    \left[\sum_{i\neq h,k}\eta^{kh}_{ki}(t)\bee{ki}+
      \sum_{j\neq h,k}\eta^{kh}_{jh}(t)\bee{jh}+\eta_{kh}^{kh}(t)\bee{kh}
      \right]
  \end{multline}
  for all $k,h$.
  So,  there is a cocycle
  in the same cohomology class as $f$ of the form~\eqref{dM-}.
  From now on (with an abuse of notation) we shall assume that $F$ has been
  chosen to represents such a cocycle, i.e. that~\eqref{dM-} is the
  original form of $F$ and $[F]=f_\phi$.

  In particular, observe that $\eta_{kk}^{ij}(t)=0$ for all $i,j=1,\dots,n$.

  We can now apply
  conditions~\eqref{ahk} and~\eqref{akp}, which give that for any $i,j\neq k,h$
  \[ \begin{cases}
    \eta^{kh}_{jh}(t)=-t\eta^{kj}_{kk}(1)=0 \\
    \eta^{kh}_{ki}(t)=-t\eta^{ih}_{hh}(1)=0.
  \end{cases} \]

Define $\eta_{kh}(t):=\eta^{kh}_{kh}(t)$ for all $k\neq h$.
This gives
\[ f=[\eta_{kh}(t)\bee{kh}] \]
as required.
\end{proof}

\begin{lemma}  \label{e7}
 Let  $\eta_{k\ell}$ be the function defined in Lemma~\ref{rn>3}. The following hold:
	\begin{enumerate}[{\rm (M1)}]
\item\label{M1} $\eta_{kh}(-s)=-\eta_{kh}(s)$;
\item\label{M2}
  $\eta_{k\ell}(st)=t\eta_{kh}(s)+s\eta_{h\ell}(t), \forall s,t\in \KK.$
\end{enumerate}
  \end{lemma}

\begin{proof}
  Recall that $F\in C^1(G_n,A_n)$ is such that $f_\phi=[F]$. Moreover, by Lemma \ref{rn>3} we can assume that $F(u_{h}(t)) = \eta_{kh}(t)$ for every choice of th generator $u_{kh}(t)$. 
\begin{enumerate}[(M1)]
\item  We have
  \begin{multline*}
    [0]=f(I_n)=f(u_{kh}(s)u_{kh}(-s))=\\
        \left[ (I_n+s\bee{kh})F(u_{kh}(s))(I_n-s\bee{kh})+F(u_{kh}(-s))\right]=\\
        \left[ (I_n+s\bee{kh})(\eta_{kh}(s)\bee{kh})(I_n-s\bee{kh})+\eta_{kh}(-s)\bee{kh}\right]=\\
        [\eta_{kh}(s)\bee{kh}+\eta_{kh}(-s)\bee{kh}],
      \end{multline*}
      whence $\eta_{kh}(-s)=-\eta_{kh}(s)$.
\item Let $k,h,\ell$ three distinct indexes with $1\leq k,h,\ell\leq n$
Then, by Property (U\ref{U4}),
\[  u_{kh}(s)u_{h\ell}(t)u_{kh}(-s)u_{h\ell}(-t)=u_{k\ell}(st). \]
Therefore, by the choice of $F$, we have
\begin{equation}\label{19}
 F(u_{kh}(s)u_{h\ell}(t)u_{kh}(-s)u_{h\ell}(-t))=\eta_{k\ell}(st)\bee{k\ell}.
\end{equation}
To ease the computations, put
\[
  \begin{cases}
    s_1(s,t):=F(u_{kh}(s)u_{h\ell}(t))  \\
    s_2(s,t):=F(u_{kh}(s)u_{h\ell}(t)u_{kh}(-s)) \\
  \end{cases}
\]
By direct computations and using (M\ref{M1}), we have
\begin{multline*}
  s_1(s,t)=F(u_{kh}(s)u_{h\ell}(t)) \equiv \\
  (I_n-t\bee{h\ell})F(u_{kh}(s))(I_n+t\bee{h\ell})+\eta_{h\ell}(t)=\\
  (I_n-t\bee{h\ell})\eta_{kh}(s)\bee{kh}(I_n+t\bee{h\ell})+
  \eta_{h\ell}(t)\bee{h\ell}=\\
  \eta_{kh}(s)\bee{kh}+\eta_{h\ell}(t)\bee{h\ell}+
  t\eta_{kh}(s)\bee{k\ell}.
\end{multline*}
\begin{multline*}
  s_2(s,t)=F(u_{kh}(s)u_{h\ell}(t)u_{kh}(-s))\equiv \\
  u_{kh}(s)F(u_{kh}(s)u_{h\ell}(t))u_{kh}(-s)+F(u_{kh}(-s))=\\
  (I_n+s\bee{kh})s_1(s,t)(I_n-s\bee{kh})-\eta_{kh}(s)\bee{kh}=\\
  s_1(s,t)+s\bee{kh}s_1(s,t)-ss_1(s,t)\bee{kh}-s^2\bee{kh}s_1(s,t)\bee{kh}
  -\eta_{kh}(s)\bee{kh} \equiv\\
  \eta_{kh}(s)\bee{kh}+\eta_{h\ell}(t)\bee{h\ell}+
  t\eta_{kh}(s)\bee{k\ell}+
  s\eta_{h\ell}(t)\bee{k\ell}-\eta_{kh}(s)\bee{kh}=\\
  \eta_{h\ell}(t)\bee{h\ell}+(t\eta_{kh}(s)+s\eta_{h\ell}(t))\bee{k\ell}.
\end{multline*}
By \eqref{19} and te above we have
\begin{multline*}
 \eta_{k\ell}(st)\bee{k\ell} = F(u_{kh}(s)u_{h\ell}(t)u_{kh}(-s)u_{h\ell}(-t)) \equiv \\
u_{h\ell}(t)s_2(s,t)u_{h\ell}(-t)+
 F(u_{h\ell}(-t))=\\
 (I_n+t\bee{h\ell})s_2(s,t)(I_n-t\bee{h\ell})-\eta_{h\ell}(t)\bee{h\ell}=\\
s_2(s,t)+t\bee{h\ell}s_2(s,t)-ts_2(s,t)\bee{h\ell}-t^2\bee{h\ell}s_2(s,t)\bee{h\ell}-\eta_{h\ell}(t)\bee{h\ell} \equiv\\
 (t\eta_{kh}(s)+s\eta_{h\ell}(t))\bee{k\ell}.
\end{multline*}
In conclusion $\eta_{k\ell}(st)\bee{k\ell} \equiv  (t\eta_{kh}(s)+s\eta_{h\ell}(t))\bee{k\ell}$. However this equivalence is actually an equality, by the Cleaning Principle. Property (M\ref{M2}) is proved.  
\end{enumerate}
\end{proof}

\begin{lemma}
  \label{canon-repr}
Let $\phi\in H^1(G_n,A_n^*)$. Then $\phi$ admits a representative $f_\phi\in Z^1(G_n,A^*_n)$ such that for all
  $k,h=1,\ldots,n$ with $k\neq h$ there exist a function
  $\alpha_{kh}:\KK\to\KK$ with $\alpha_{kh}(1)=0$ and such
  that 
  \begin{equation}\label{H1}
    f_{\phi}(u_{kh}(t))=[\alpha_{kh}(t)\bee{kh}],\,\,\, \forall t\in \KK. 
  \end{equation}
\end{lemma}
\begin{proof}
By Lemma~\ref{rn>3} there exist a representative $f_\phi\in Z^(G_n,A^*_n)$ of $\phi$ such that for every choice of distinct indexes $k, h\in\{1, 2,..., n\}$ there exists and a function $\eta_{kh}\colon \KK\rightarrow \KK$ such that
$f_\phi(u_{kh}(t))=[\eta_{kh}(t)\bee{kh}]$ for every $t\in \KK$. 

 We claim that there exists a coboundary $w_S\in B^1(G,A_n^*)$ satisfying $w_S(u_{kh}(1))=[-\eta_{kh}(1)\bee{kh}]$ for every choice of distinct indexes $k$ and $h$.

Suppose $S$ is a diagonal $n\times n$ matrix $S=\mathrm{diag}(s_1,s_2,\dots,s_n)$ and define
 \[w_S\colon G_n\rightarrow A_n^*,\,\,w_S(u_{kh}(t))=[t(s_h-s_k)\bee{kh}].\]
 Note that, as in the the first part of the proof of Lemma~\ref{rn>3}, the coboundary $w_S$ is well defined and does not depend on the representative $S$ of $[S]$ in $\cM_n/\cI_n.$

	Put $S:=\mathrm{diag}(0,\eta_{2,1}(1),\eta_{3,1}(1)\dots, \eta_{n,1}(1))$, i.e. $s_1=0$ and $s_k=\eta_{k,1}(1)$ for $k>1.$
 By (M\ref{M2}) of Lemma~\ref{e7} with $\ell=1$, we have $s_k-s_h=\eta_{k,1}(1)-\eta_{h,1}(1)=\eta_{kh}(1)$, 
 whence $w_S(u_{kh}(1))=[(s_h-s_k)\bee{kh}]=[-\eta_{kh}(1) \bee{kh}]$ and the claim is proved. Put
\begin{equation}\label{alpha}
\alpha_{kh}(t):=\eta_{kh}(t)-\eta_{kh}(1)
\end{equation}
 and $f_{\phi,S}:= f_\phi+w_S$; so $f_{\phi,S}\in Z^1(G_n,A^*_n)$ still represents $\phi$ and $f_{\phi,S}(u_{kh}(t))=[\alpha_{kh}(t)\bee{kh}]$.
According to definition \eqref{alpha}, we have $\alpha_{kh}(1) = 0$. 
 \end{proof}

\begin{lemma}
  \label{add}
The function $\alpha_{kh}\colon \KK\rightarrow \KK$ introduced in Lemma~\ref{canon-repr} is additive:
  \[ \alpha_{kh}(s+t)=\alpha_{kh}(s)+\alpha_{kh}(t),\,\,\forall s,t\in\KK. \]
\end{lemma}
\begin{proof}
We can choose a representative $F\in C^1(G_n,A_n)$ of $\phi$ in such a way that for every choice of distinct $k, h\in\{1,2,..., n\}$ an every $t\in \KK$ we have $F(u_{kh}(t))=\alpha_{kh}(t)\bee{kh}$. Hence $F(u_{kh}(1)) = 0$, because $\alpha_{kh}(1)=0$.
Since  $u_{kh}(s)u_{kh}(t)= u_{kh}(s+t)$, we have
\begin{multline}
 \alpha_{kh}(s+t)\bee{kh}=F(u_{kh}(s+t))=F(u_{kh}(s)u_{kh}(t)) \equiv \\
  u_{kh}(-t)F(u_{kh}(s))u_{kh}(t)+F(u_{kh}(t))= \\
  (I_n-t\bee{kh})(\alpha_{kh}(s)\bee{kh})(I_n+t\bee{kh})+\alpha_{kh}(t)\bee{kh}=\\
  (\alpha_{kh}(s)+\alpha_{kh}(t))\bee{k,h}.
\end{multline}
So $\alpha_{kh}(s+t)\bee{kh} \equiv (\alpha_{kh}(s)+\alpha_{kh}(t))\bee{k,h}$, namely  $\alpha_{kh}(s+t)\bee{kh} = (\alpha_{kh}(s)+\alpha_{kh}(t))\bee{k,h}$ by the Cleaning Principle.  \end{proof}

The next lemma completes the proof of Lemma \ref{main4} under the assumption that ($n\geq 4$ and) either $\KK \neq \FF_2$ or $n > 4$. 

\begin{lemma}
  \label{uni}
Let $\phi\in H^1(G_n,A_n^*)$. Then there exists a function
  $\alpha_{\phi}:\KK\to\KK$ with $\alpha_{\phi}(1)=0$ and a representative $f_\phi\in Z^1(G_n,A_n^*)$ of $\phi$ such
  \begin{equation}\label{H1-nuova}
    f_\phi(u_{kh}(t))=[\alpha_{\phi}(t)\bee{kh}],\,\,\, \forall k,h\in\{1,2,...,n\}, ~k\neq h, ~ \forall t\in \KK. 
  \end{equation}
Moreover,  $\alpha_{\phi}$ is a derivation of $\KK.$
   \end{lemma}
\begin{proof}
By Lemma~\ref{canon-repr} and Lemma \ref{add} we already know that $\phi$ admits a representative $f_\phi\in Z^1(G_n,A^*_n)$ such that for every choice of distinct indexes $k$ and $h$ there exists an additive function 
  $\alpha_{kh}:\KK\to\KK$ such that $\alpha_{kh}(1)=0$ and
  \[ f_\phi(u_{kh}(t))=[\alpha_{kh}(t)\bee{kh}],\,\,\, \forall t\in \KK.\]
We shall now prove that $\alpha_{kh}= \alpha_{\ell m}$ for every choice of $k,h,\ell, m\in\{1,2,..., n\}$, $k\neq h$, $\ell\neq m$. 
Let $k,h,\ell$ be pairwise distinct indexes. By (M\ref{M2}) of Lemma~\ref{e7} with $s=1$, we have
\[ \alpha_{k\ell}(t)=t\alpha_{kh}(1)+1\alpha_{h\ell}(t)=\alpha_{h\ell}(t)\]
and, likewise for $t=1$,
\[ \alpha_{k\ell}(s)=1\alpha_{kh}(s)+s\alpha_{h\ell}(1)=\alpha_{kh}(s). \]
This means that $\alpha_{kh}=\alpha_{k\ell}=\alpha_{h\ell}$.
As $k,h,\ell$ vary in all possible ways in ${\{1,\dots, n\}}\choose {3}$ and $n > 3$ by assumption, this implies that there is
a function $\alpha_{\phi}:\KK\to\KK$ such that
$\alpha_{\phi}(s):=\alpha_{kh}(s)$ for all $k\neq h$.

By Lemma~\ref{canon-repr},  since $\alpha_{kh}(1)=0$, we have
$\alpha_{\phi}(1)=0$. By Lemma~\ref{add} we have $\alpha_{\phi}(s+t)=\alpha_{\phi}(s)+\alpha_{\phi}(t)$
and by (M\ref{M2}) of Lemma~\ref{e7} we have
$\alpha_{\phi}(xy)=\alpha_{\phi}(x)y+x\alpha_{\phi}(y)$.
So $\alpha_{\phi}$ satisfies conditions (D\ref{D1}) and (D\ref{D2}) of Section~\ref{derivazioni}, i.e. it is
a derivation of $\KK$.
\end{proof}

\subsection{Case $\KK=\FF_2$ and $n = 4$.}\label{caso n=4 e K=F_2}

When $\KK = \FF_2$ and $n = 4$ the conditions of Table \ref{tab-c} fail to force sufficiently many entries of $F(u_{kh}(t))$ to be null. However in this case it is possible to perform an exhaustive computer search as follows.
The group $\SL(4,\FF_2)$ is $2$-generated.
Using~\cite{GAP4},
\lstset{language=gap,
	basicstyle=\small}
\begin{lstlisting}
n:=4; q:=2;
G:=SL(n,q);
ARing:=GF(q)^[n,n];
Gens:=GeneratorsOfGroup(G);
\end{lstlisting}
we see that the two matrices
\[ g_1:=\begin{pmatrix}
  1 & 1 & 0 & 0 \\
  0 & 1 & 0 & 0 \\
  0 & 0 & 1 & 0 \\
  0 & 0 & 0 & 1 \\
\end{pmatrix},\qquad
g_2:=\begin{pmatrix}
  0 & 0 & 0 & 1 \\
  1 & 0 & 0 & 0 \\
  0 & 1 & 0 & 0 \\
  0 & 0 & 1 & 0
\end{pmatrix} \]
are a set of generators for $\SL(4,\FF_4)$.
It is now possible to express a presentation of the group in
terms of $g_1$ and $g_2$ as follows:
\begin{lstlisting}
IsoG:=IsomorphismFpGroupByGenerators(G,Gens);
FpG:=Image(IsoG);
P:=PresentationViaCosetTable(FpG);
SimplifyPresentation(P);
f1:=Gens[1];
f2:=Gens[2];
Gp:=FpGroupPresentation(P);
R:=RelatorsOfFpGroup(Gp);
TForm:=List(R,x->TietzeWordAbstractWord(x));
\end{lstlisting}
This gives the following $8$ relators in terms of the
elements $g_1$ and $g_2$:
\[ g_1^2,\quad g_2^4,\quad (g_1g_2^{-2})^4,\quad (g_1g_2g_1g_2^{-1})^4,
 \quad (g_1g_2)^{15},
\]
\[ g_1g_2g_1g_2^{-1}g_1g_2^2g_1g_2^{-1}(g_1g_2g_1g_2^{-2}g_1g_2)^2g_1g_2g_1g_2^{-2},
\]
\[ g_1g_2^{-1}g_1g_2^2g_1g_2^{-1}(g_1g_2^-2g_1g_2)^4g_1g_2^{-2}, \]
\[  g_2^{-1}g_1g_2^{-1}(g_1g_2)^4(g_1g_2^{-1})^4(g_1g_2)^3(g_2g_1g_2^{-1}g_1)^2g_2^{-1}g_1.
\]
A cocycle $f\in Z^1(\SL(4,2),A^*)$ is uniquely determined by the values of $f(g_1)$ and $f(g_2)$ in $A^*$.
Clearly $f$ must also satisfy the condition that for any relator $r$, $f(r)=f(1)=0$.

The function {\tt fProd} has as input three lists: ${\tt elts}=( g_1, g_2,\dots, g_k)$ a list of elements of $\SL(4,\FF_4)$,
${\tt f}:=( f(g_1), f(g_2),\dots f(g_k) )$ a list corresponding to the values of a cocycle $f$ on the elements of {\tt elts} and
${\tt Idx}:=(i_1,i_2)$, a pair of indexes $i_1,i_2$. It returns the value of $f(g_{i_1}g_{i_2})$.
\begin{lstlisting}
fInv:=function(el,f)
# (el:element, f:cocycle)->f(el^-1)
    return -el*f*el^-1;
end;

fProd:=function(elts,f,Ix)
    local lElts,lF,i,Idx;
    Idx:=ShallowCopy(Ix);
    lElts:=ShallowCopy(elts);
    lF:=ShallowCopy(f);
    for i in [1..Size(Idx)] do
        if Idx[i]<0 then
            Idx[i]:=-Idx[i];
            lF[Idx[i]]:=fInv(lElts[Idx[i]],lF[Idx[i]]);
            lElts[Idx[i]]:=lElts[Idx[i]]^-1;
        fi;
    od;
    return lElts[Idx[2]]^(-1)*lF[Idx[1]]*lElts[Idx[2]]+lF[Idx[2]];
end;
\end{lstlisting}

The function {\tt ApplyRed} has three inputs: a list {\tt Mats} of matrices of $\SL(4,2)$, the values
{\tt Vals} of a cocycle $f$ takes on the elements of {\tt Mats} and a list {\tt Strng} which describes
how the matrices in {\tt Mats} must be multiplied together. Its output is the value of the cocycle
in the matrix obtained by applying {\tt Strng} to {\tt Mats}.
\begin{lstlisting}
myFilter:=function(LST,t)
    if t>0 then
        return LST[t];
    else
        return LST[-t]^-1;
    fi;
end;

ApplyRed:=function(Mats,Vals, Strng )
   local s,LMats,LVals,LStrng,
          nMat,nVal;
    LMats:=ShallowCopy(Mats);
    LVals:=ShallowCopy(Vals);
    LStrng:=ShallowCopy(Strng);    
    s:=Length(LStrng);
    if s=1 then
        return Vals[Strng[1]];
    fi;
    nMat:=myFilter(LMats,Strng[s-1])*myFilter(LMats,Strng[s]);
    nVal:=fProd(LMats,LVals,Strng{[s-1,s]});
    Add(LMats,nMat);
    Add(LVals,nVal);
    LStrng[s-1]:=Size(LMats);
    LStrng:=LStrng{[1..s-1]};
    return ApplyRed(LMats,LVals,LStrng);
end;
\end{lstlisting}
It is now possible to enumerate all possible cocycles in $Z^1(\SL(2,\FF_4),A^*)$
by considering all mappings $f:\{g_1,g_2\}\to A^*$ such that for any relator $r$
we have $f(r)=0$.

This is done as follows:
\begin{lstlisting}
#Ideal defining A* 
Idd:=[ NullMat(n,n,GF(q)), IdentityMat(n,GF(q)) ];

IsZeroAst:=function(A)
 return A in Idd;
end;

# Filter on relators which contain only one generator 
ARing1:=Filtered(ARing,x->
    IsZeroAst(ApplyRed(Gens,[x,0],TForm[1])));;
ARing2:=Filtered(ARing,x->
    IsZeroAst(ApplyRed(Gens,[0,x],TForm[2])));;

ValuesToConsider:=Cartesian(ARing1,ARing2);;

# Filter on remaining relators
for i in [3..Size(TForm)] do
    ValuesToConsider:=Filtered(ValuesToConsider,
              x->IsZeroAst(ApplyRed(Gens,x,TForm[i])));;
end;

# Group values
ExpandValues:=function(X)
  return Set(Cartesian(Idd,Idd),t->X+t);
end;

MyCocycles:=Set(ValuesToConsider,u->ExpandValues(u));
MyCoboundaries:=Set(ARing,a->
        ExpandValues( [f1^-1*a*f1-a, f2^-1*a*f2-a]));;

\end{lstlisting}
After this exhaustive computer search {\tt ValuesToConsider} contains exactly $4$ elements for any
cocycle. The final step identifies cocycles in {\tt MyCocycles} up to the equivalence relation defining $A^*$.

Keeping in mind that every element  $a\in A^*$ defines a different $1$-coboundary
by the map $f_a(g)=g^{-1}ag-a$, we also initialize the list {\tt MyCoboundaries} with the representatives
of the $1$-coboundary.
We can now see directly that
\[ Z^1(\SL(4,\FF_2),A^*)=B^1(\SL(4,\FF_2),A^*).\]
Consequently,   $H^1(\SL(4,\FF_2),A^*) = \{0\}$. As $\Der(\FF_2)$ is null as well, the isomorphism
$H^1(\SL(4,\FF_2),A^*) \cong \Der(\FF_2)$ trivially holds. 

Authors' addresses:
\vskip.4cm\noindent\nobreak
\begin{minipage}[t]{8cm}
\small{Ilaria Cardinali, Antonio Pasini\\
Dep. Information Engineering and Mathematics \\University of Siena\\
Via Roma 56, I-53100 Siena, Italy\\
ilaria.cardinali@unisi.it \\  antonio.pasini@unisi.it}
\end{minipage}
\hfill
\begin{minipage}[t]{6cm}
\small{Luca Giuzzi\\
D.I.C.A.T.A.M. \\
University of Brescia\\
Via Branze 43, I-25123 Brescia, Italy \\
luca.giuzzi@unibs.it}
\end{minipage}

\end{document}